\documentclass[11pt]{amsart}

\usepackage{color}
\usepackage[TS1,T1]{fontenc}
\usepackage[latin1]{inputenc}
\usepackage{amsfonts,amssymb,amsmath,amsthm,enumerate,latexsym}
\usepackage{hyperref}

\newcommand{\R}{\mathbb{R}}

\newcommand{\defeq}{\mathrel{\mathop:}=}

\def\R{{\mathbb {R}}}

\newtheorem{teo}{Theorem}[section]

\newtheorem{corol}[teo]{Corollary}

\theoremstyle{remark}
\newtheorem{remark}[teo]{Remark}

\theoremstyle{definition}
\newtheorem{defi}[teo]{Definition}

\numberwithin{equation}{section}

\begin{document}

\title[Fractional elliptic problems with nonlinear gradient sources and measures]
{Fractional elliptic problems with nonlinear gradient sources and measures}

\author[J.V. da Silva, P. Ochoa and A. Silva]{Jo\~{a}o Vitor da Silva, Pablo Ochoa and Anal\'ia Silva}

\address{J.V da Silva
\hfill\break\indent Departamento de Matem\'atica - Instituto de Ci\^{e}ncias Exatas.
 \hfill\break\indent Universidade de Bras\'{i}lia.
 \hfill\break\indent
  Campus Universit\'{a}rio Darcy Ribeiro, 70910-900.
 \hfill\break\indent Bras\'{i}lia - DF - Brazil.}
 \email{{\tt J.V.Silva@mat.unb.br}}


\address{P. Ochoa
\hfill\break\indent Facultad de Ingenier\'ia.
 \hfill\break\indent Universidad Nacional de Cuyo and CONICET.
 \hfill\break\indent
  Ciudad Universitaria - Parque General San Mart\'in .
 \hfill\break\indent 5500 Mendoza, Argentina.}
 \email{{\tt ochopablo@gmail.com }}

\address{A. Silva
\hfill\break\indent Instituto de Matem\'atica Aplicada San Luis, IMASL.
 \hfill\break\indent Universidad Nacional de San Luis and CONICET.
 \hfill\break\indent
Ej\'{e}rcito de los Andes 950.
 \hfill\break\indent D5700HHW San Luis, Argentina.}
 \email{{\tt asilva@dm.uba.ar }}
\urladdr[A.Silva]{https://analiasilva.weebly.com/}


\begin{abstract}
In this manuscript we deal with existence/uniqueness and regularity issues
of suitable weak solutions to nonlocal problems driven by
fractional Laplace type operators. Different from previous
researches, in our approach we consider gradient non-linearity
sources with subcritical growth, as well as appropriated measures as
sources and boundary datum. We provide an in-depth discussion on the
notions of solutions involved together with existence/uniqueness results in
different regimes and for different boundary value problems.
Finally, this work extends previous ones by dealing with more
general nonlocal operators, source terms and boundary data.
\end{abstract}

\maketitle

\section{Introduction}

\subsection{Main proposals and contrasts with former results}

In this article, we propose to study the existence/uniqueness and regularity of appropriate weak solutions for nonlocal quasi-linear problems involving measures, more precisely we consider
\begin{equation}\label{maineq}
\left\{
\begin{array}{rcll}
  (-\Delta)^\alpha u(x) & = &g(x, |\nabla u|) +\sigma \nu &  \text{in } \Omega,\\
  u(x) & = & \varrho\mu &  \text{in } \R^N\setminus\Omega,
\end{array}
\right.
\end{equation}
where $\varrho, \sigma \geq 0$, $\mu$ and $\nu$ are suitable Radon measures, $g: \Omega \times [0, \infty) \to [0, \infty)$ is a continuous function fulfilling certain growth conditions (to be presented \textit{a posteriori}) and $\Omega \subset \R^N$ is a $C^{2}$ bounded domain.

As the nonlinear term $g$ appears in the right-hand side, such a model \eqref{maineq} with $\varrho =0,$ may be understood as a Kardar-Parisi-Zhang stationary problem (models of growing interfaces) driving by fractional diffusion (see \cite{KPZ} for the model in the local setting and \cite{AP} for an instrumental work in the nonlocal scenario). On the other hand, the problem with the nonlinear (Hamiltonian) term in the left-hand side is the stationary counterpart of a Hamilton-Jacobi equation with a viscosity term under certain critical fractional diffusion (see \cite{Sil} and the references therein).

In the last two decades, the fractional Laplacian operator $\mathcal{L} = (-\Delta)^{\alpha}$, or more general elliptic linear integro-differential operators (with singular kernels), have been a classic topic of research in several fields of pure mathematics such as Geometry, Harmonic Analysis, PDEs and Probability. Furthermore, there has been renewed interest in these kind of operators due to their current connections with certain stochastic processes of L\`{e}vy type \cite{A04}, \cite{BGR61}, \cite{BKN02}, \cite{ChenS}, \cite{Kul97}, \cite{L00} theory of semigroups \cite{GMS13}, \cite{Stinga19}, recent progress in geometric analysis and conformal geometry \cite{CG11}, \cite{FV13}, \cite{GQ13}, and existence and regularity issues in a number of nonlocal diffusion and free boundaries problems \cite{BFR-O}, \cite{CR-OS17}, \cite{CS}, \cite{CS19}, \cite{FBSS18}, \cite{FBSSp19}, \cite{RS}, \cite{S} and \cite{Sil07}, just to mention a few.

We should also highlight that nonlocal type operators arise naturally in a number of applied mathematical modelling such as in continuum mechanics, image processing, crystal dislocation, Nonlinear Dynamics (Geophysical Flows), phase transition phenomena, population dynamics, nonlocal optimal control and game theory as pointed out in \cite{BCF}, \cite{BCF12}, \cite{BV}, \cite{C}, \cite{CV10}, \cite{CV}, \cite{DFV}, \cite{DPV}, \cite{GO}, \cite{L00} and the references therein. Just for illustration, the fractional heat equation may appear in probabilistic random-walk procedures and, in turn, the stationary case may do so in pay-off models (see \cite{BV} and the references therein). In the works \cite{MK} and \cite{MK1} the description of anomalous diffusion via fractional dynamics is investigated and various fractional partial differential equations are derived from L\`{e}vy random walk models, extending Brownian motion models in a natural way. Finally, fractional type operators are also encompassed in mathematical modeling of financial markets, since L\`{e}vy type processes with jumps take place as more accurate models of stock pricing (see e.g. \cite{A04} and \cite{CT} for some illustrative examples). In fact, the \textit{boundary condition}
$$
   u=\varrho \mu \,\,\textnormal{ in } \,\mathbb{R}^{N}\setminus \Omega
$$
which is given in the whole complement may be interpreted from the stochastic point of view as the fact that a L\`{e}vy process can exit the domain $\Omega$ for the first time jumping to any subset $E \subset \mathbb{R}^{N}\setminus \Omega$ with \textit{probability density} given by $\varrho\mu(E)$.

As a prelude to our investigations, let us present a historical
overview regarding recent advances in semi-linear elliptic problems
with measure data. In the pioneering work \cite{Bre} (see also
\cite{BB}), Brezis studied the existence/uniqueness of solutions to
semi-linear Dirichlet elliptic problems of the form
\begin{equation*}
\begin{cases}
-\Delta u +g(u) = \nu &  \text{in } \Omega,\\
\qquad\qquad\,\,\,\,\, u =0 &  \text{in } \partial\Omega,
\end{cases}
\end{equation*}
where $\nu$ is a bounded measure, $g$ is non-decreasing, positive and satisfies the integral growth condition
$$
  \int_{1}^{\infty} \frac{(g(s)-g(-s))}{s^{2\frac{N-1}{N-2}}}ds< \infty.
$$
Observe that when $g$ is a $p-$th power, i.e. $g(s)=s^{p}$, the above integrability condition is satisfied whenever
$$
  p< \frac{N}{N-2}.
$$
When $p \geq \frac{N}{N-2}$, solutions might not exist (see for instance \cite{BB}).

Posteriorly, V\'eron generalizes the former results in \cite{Ve4} by replacing the Laplacian by more general second-order (uniformly) elliptic operators (in divergence form), allowing for measure sources so that
$$
  \int_\Omega \rho^{\beta}(x)d|\nu| < \infty,
$$
for $\beta\in [0, 1]$ and where $\rho$ is the distance-to-the-boundary  function. The non-linearity $g$ now is assumed to satisfy the integrability condition
$$
  \int_{1}^{\infty} \frac{(g(s)-g(-s))}{s^{2\frac{N+\beta-1}{N+\beta-2}}}ds< \infty.
$$
Finally, in \cite{NV}, Nguyen-Phuoc and V\'eron obtained existence results of solutions to
\begin{equation*}
\begin{cases}
-\Delta u +g(|\nabla u|) = \nu &  \text{in } \Omega,\\
\qquad\qquad\qquad u =0 &  \text{in } \partial\Omega,
\end{cases}
\end{equation*}
where $g$ fulfils the integrability assumption
$$
   \int_{1}^{\infty} \frac{g(s)}{s^{\frac{2N-1}{N-1}}}ds< \infty.
$$
For an instrumental survey on elliptic Dirichlet problems involving the Laplace operator and measures we recommend the Marcus and V\'{e}ron's Monograph \cite{Marcus-V} and the references therein.

 Now, let us highlight some pivotal works regarding existence and regularity for problems driven by fractional diffusion with measure datum (for $g\equiv0$)
 $$
 -\mathcal{L} u = \nu \quad \text{in} \quad \Omega \subset \R^N.
 $$
 Such results for the fractional Laplacian (for powers $\alpha \in \left(\frac{1}{2}, 1\right)$) have been obtained in \cite{KPU}, where the approach is via duality method. In the recent work \cite{CV1} the authors deal with fractional equations (this time for any $\alpha \in (0, 1)$) involving measures, where the study is carried out through fundamental solutions. In \cite{AAN10} is proposed a notion of re-normalised solution for semi-linear equations. The work \cite{KMS15} (see also the enlightening survey \cite{KMS18}) deals with more general nonlinear integro-differential equations (possibly degenerate or singular) with measurable, elliptic/coercive and symmetric kernels, thereby obtaining existence of suitable weak solutions (SOLA - Solutions Obtained as Limits of Approximations) and regularity results by means of nonlinear potentials of Wolff type.

Concerning elliptic problems with measure source governed by
fractional Laplacian (for $g\not\equiv 0$), recently Chen and
V\'{e}ron in \cite{CV1} investigated the semi-linear fractional
equation
$$
\begin{cases}
(-\Delta)^{\alpha} u + g(u) = \nu &  \text{in } \Omega\\
\qquad\quad u =0 &  \text{in } \mathbb{R}^{N}\setminus \Omega,
\end{cases}
$$
where $\nu  \in \mathcal{M}(\Omega, \rho^{\beta})$, i.e., $\displaystyle \int_\Omega \rho^{\beta}(x)d|\nu| < \infty,$ with $0\le \beta \le \alpha$. In such a work the authors proved existence/uniqueness of solutions when $g$ is nondecreasing and satisfies
$$
   \int_{1}^{\infty} \frac{g(s)-g(-s)}{s^{1+k_{\alpha, \beta}}}ds< \infty,
$$
where
$$
k_{\alpha, \beta} = \left\{
\begin{array}{rcl}
  \frac{N}{N-2\alpha} & \mbox{if} & \beta \in \left[0, \frac{N-2\alpha}{N}\alpha\right]\\
  \frac{N+\alpha}{N-2\alpha+\beta} & \mbox{if} &  \beta \in \left(\frac{N-2\alpha}{N}\alpha, \alpha\right],
\end{array}
\right.
$$

With respect to fractional Laplacian with gradient source term,
according to our scientific knowledge up to date, the more recent
findings regarding existence/uniqueness and regularity issues of
solutions to problems like \eqref{maineq} can be found in \cite{AP},
\cite{VC1} and \cite{CV} respectively. In \cite{AP} {(see also \cite{AP20} for a corrigendum) Abdellaoui and Peral address an extensive and complete analysis to
\begin{equation*}
\begin{cases}
(-\Delta)^{\alpha} u = |\nabla u|^q + \lambda f &  \text{in } \Omega,\\
\qquad\quad u =0 &  \text{on } \mathbb{R}^{N}\setminus \Omega\\
\qquad\quad u >0 &  \text{in } \Omega,
\end{cases}
\end{equation*}
regarding existence/uniqueness and regularity of weak solutions in three different cases: subcritical, $1<q<2\alpha$; critical, $q=2\alpha$; and supercritical, $q>2\alpha$, for $\alpha \in \left(\frac{1}{2}, 1\right) $ and $f$ a measurable non-negative function with suitable integrability hypotheses. On the other hand, in \cite{CV}, the authors treated the problem
\begin{equation*}
\begin{cases}
(-\Delta)^{\alpha} u = \pm g(|\nabla u|) + \nu &  \text{in } \Omega,\\
\qquad\quad u =0 &  \text{in } \mathbb{R}^{N}\setminus \Omega,
\end{cases}
\end{equation*}
while the case for prescribed measures in $\mathbb{R}^{N}\setminus
\Omega$ was considered in \cite{VC1}. In both cases, the
non-linearity is assumed to satisfy an integral or polynomial growth
condition.

These former results have been our starting point in obtaining
qualitative results for models like \eqref{maineq} under
appropriated assumption on the data, and with nonlinear gradient
sources and measures.

In order to finish these theoretical landmarks, let us briefly
present the more current existence/uniqueness results related to
measure supported on the boundary. In this direction, in \cite{CH}
the authors studied weak solutions of the fractional elliptic
problem
\begin{equation}\label{Eq1.2}
\left\{
\begin{array}{rcll}
  (-\Delta)^\alpha u(x) + \varepsilon g(u)& = & k \frac{\partial \nu}{\partial \overrightarrow{n}_x^{\alpha}} &  \text{in } \overline{\Omega},\\
  u(x) & = & 0 &  \text{in } \R^N\setminus\overline{\Omega},
\end{array}
\right.
\end{equation}
where $k>0$, $\varepsilon = \pm 1$, $\nu$ is a bounded Radon measure supported in $\partial \Omega$ and $\frac{\partial \nu}{\partial \overrightarrow{n}_x^{\alpha}}$ is defined in a suitable distributional sense (see Section \ref{Sec6} for more detail). In such a context, they prove (for $\varepsilon =1$) that \eqref{Eq1.2} admits a unique weak solution when $g$ is a continuous nondecreasing function satisfying the integral condition
$$
  \displaystyle \int_{1}^{+\infty} \frac{(g(s)-g(-s))}{s^{1+\frac{N+\alpha}{N-\alpha}}}ds< +\infty.
$$
On the other hand, when $\varepsilon = -1$ and $\nu$ is nonnegative, by employing the Schauder's fixed point theorem, they obtain existence of a positive solution under the hypothesis that $g$ is a continuous function satisfying:
$$
  \displaystyle \int_{1}^{+\infty} \frac{g(s)}{s^{1+\frac{N+\alpha}{N-\alpha}}}ds< +\infty.
$$

In contrast with \cite{CH}, in our approach the boundary term $\frac{\partial \nu}{\partial \overrightarrow{n}_x^{\alpha}}$ will just appear (in a natural way) when we invoke the nonlocal integration by parts criterium for our definition of solution. Furthermore, we will focus our attention in proving existence results to problems like \eqref{maineq} where $\nu$ is supported on the boundary.

\subsection{Our main contributions}

In this work, we propose to study problem \eqref{maineq} with a non-linearity $g$ depending on both the spatial and gradient variables. Roughly speaking, it will be assumed that $g$ is continuous,  verifies a polynomial growth in $|\nabla u|$ and it is integrable in $x$. Therefore, the main contributions of our work will be:
\begin{enumerate}
\item A detailed discussion of the appropriate notion of distributional solutions to elliptic integro-differential problems involving measures as both: sources and Dirichlet boundary data.
\item Existence  of solutions in  two different regimes based on different ranges for a $p-$growth type of $g$ w.r.t. $|\nabla u|$:
\begin{enumerate}
\item sub-linear regime: $0 < p \leq 1$.
\item super-linear and sub-critical: $1 <p<p^{\ast} \defeq \frac{N}{N-(2\alpha-1)}$. We also state uniqueness in this case.
\end{enumerate}
\item Stability of solutions under perturbations of the data.
\item Extension of the analysis to existence of solutions  to boundary value problems with measures concentrated on $\partial \Omega$ and $\mathbb{R}^{N}\setminus \overline{\Omega}$.
\item Discussion to more general fractional type operators.
\end{enumerate}

Let us discuss heuristically the role played for the critical exponent $p^{\ast} = \frac{N}{N-(2\alpha-1)}$.  First, observe that if $u$ solves
 \begin{equation*}
\begin{cases}
(-\Delta)^{\alpha} u = \nu &  \text{in } \Omega,\\
\qquad\quad u =0 &  \text{in } \mathbb{R}^{N}\setminus \Omega,
\end{cases}
\end{equation*}
for a bounded measure $\nu$ in $\Omega$, then the following Green representation for $u$ holds
$$
  u(x)=\int_\Omega G_\alpha(x, y)d\nu(y)
$$
where $G_\alpha$ is the Green kernel for the fractional Laplacian in $\Omega$ (see \cite{VC1}, \cite{CV1} and \cite{CV}). Regarding this matter, let us remind that Bogdan-Kulczycki-Nowak, and Bogdan-Jakubowskiin in \cite{BKN02} and \cite{BJ} (see also \cite[Lemma 2.10]{AP}), by applying a probabilistic approach, were able to prove the following (point-wise) estimates of the Green function (and its gradient) provided $\frac{1}{2}<\alpha<1$:
$$
\displaystyle |G_{\alpha}(x, y)|\leq C_1\min\left\{\frac{1}{|x-y|^{N-2\alpha}}, \frac{\mathrm{d}^{\alpha}(x)}{|x-y|^{N-\alpha}}, \frac{\mathrm{d}^{\alpha}(y)}{|x-y|^{N-\alpha}}\right\}
$$
and
$$
\displaystyle  |\nabla_x G_{\alpha}(x, y)|\leq C_2G_{\alpha}(x, y)\max\left\{\frac{1}{|x-y|}, \frac{1}{\mathrm{d}(x)}\right\},
$$
where $\displaystyle \mathrm{d}: \overline{\Omega} \to \R_{+}$ is the distance function to the boundary of $\Omega$. Particularly, by \cite[Lemma 2.4]{AP20}
$$
\displaystyle  |\nabla_x G_{\alpha}(x, y)|\leq \frac{C_3}{|x-y|^{N-(2\alpha-1)}\mathrm{d}(x)^{1-\alpha}},
$$
where $C_i$, with $i = 1, \ldots, 3$, are universal constants independent of $x$ and $y$. As a consequence (see e.g. \cite[Lemma 2.12]{AP20}), $|\nabla u|\mathrm{d}^{1-\alpha} \in L^{q}(\Omega)$ for all $q<p^{\ast}$. Nevertheless, under appropriate assumptions on the growth of $g$, we shall prove in Section \ref{subcritical} that solutions to \eqref{maineq} have the regularity $W_0^{1, q}(\Omega)$ for all $q < p^{\ast}$.

For the regime $p\geq p^{\ast}$, we are not able to appeal to estimates or compactness properties of the Green operator and hence a different approach has to be undertaken (see \cite[Theorem 1.2.2]{Marcus-V} for the local case and compare with \cite{KMS15} and \cite{KMS18} for the nonlocal case when $g \equiv 0$).

Finally, we recommend to interested reader \cite{BGR61}, \cite{ChenS}, \cite{Kul97} for several properties of Green function/Poisson Kernel of certain symmetric $\alpha-$stable processes in domains via probabilistic approaches, and \cite{Bucur19} for a self-contained expository survey (without probabilistic methods) on the representation formula for the Green function on the ball.

\subsection*{Organization of the paper} In Section \ref{prelim}, we provide the basic definitions and assumptions used throughout the work. Moreover, we give a deep discussion and motivation of the notion of solutions for \eqref{maineq}. In Section \ref{subcritical}, we deal with existence of solutions in the sub-critical framework. Some stability results are provided in Section \ref{Stability}. A discussion of boundary value problems with the addition of measures concentrated on the boundary of $\Omega$ is supplied in Section \ref{Sec6}. We closed the paper with Section \ref{ClosRem} with some remarks for more general fractional type operators.

\section{Preliminaries and initial insights into the theory}\label{prelim}

In order to introduce an appropriate notion of distributional solutions, we will present some useful definitions. For $\alpha \in (0, 1)$ and $u: \mathbb{R}^{N}\to \mathbb{R}$, the fractional Laplacian $(-\Delta)^{\alpha}$ is given by
$$
   (-\Delta)^{\alpha}u(x) \defeq \lim_{\epsilon \to 0}(-\Delta)^{s}_{\epsilon}u(x)
$$
where
$$
  (-\Delta)^{\alpha}_\epsilon u(x)\defeq C_{N, \alpha}\int_{\mathbb{R}^{N}}\frac{u(x)-u(y)}{|x-y|^{N+2\alpha}}\chi_{\epsilon}(|x-y|)dy
$$
with
\begin{equation*}
\chi_t(|x|) \defeq \left\lbrace
  \begin{array}{l}
  0,  \quad |x|< t\\
    1,  \quad |x|\geq t,
  \end{array}
  \right.
\end{equation*}
and
$$
\displaystyle C_{N, \alpha} = \left(\int_{\R^N} \frac{1-\cos(\xi_1)}{|\xi|^{N+2\alpha}}d \xi\right)^{-1} = -\frac{2^{2\alpha}\Gamma(\frac{N}{2}+\alpha)}{\pi^{\frac{N}{2}}\Gamma(-\alpha)}
$$
being a normalization constant to have the following identity:
$$
  \mathcal{F}^{-1}\left(|\xi|^{2\alpha}\mathcal{F} u\right) = (-\Delta)^{\alpha}u, \quad \forall \,\,\xi\in \R^N
$$
in $\mathcal{S}\left(\R^N\right)$ (the class of Schwartz functions),
where $\mathcal{F}$ states the Fourier transform (see
\cite{Stinga19}).

Now, we will introduce the appropriated test functions space.

\begin{defi}
  We say that a function $\phi \in C^0(\mathbb{R}^{N})$ belongs to $\mathbb{X}_\alpha(\Omega)$ if and only if the following holds:
\begin{enumerate}
\item $\text{supp}(\phi) \subset \overline{\Omega}$.
\item The fractional Laplacian $(-\Delta)^{\alpha}\phi(x)$ exists for all $x \in \Omega$ and there is $C>0$ so that
    $$
      |(-\Delta)^{\alpha}\phi(x)|\leq C.
    $$
\item There are $\varphi \in L^1(\Omega)$ and $\epsilon_0 >0$ so that
$$
   |(-\Delta)^{\alpha}_\epsilon \phi(x)|\leq \varphi(x),
$$
a.e. in $\Omega$ and for all $\epsilon \in (0, \epsilon_0)$.
\end{enumerate}
\end{defi}

From now on, we denote by $G_\alpha$ the \textit{Green kernel} of $(-\Delta)^{\alpha}$ in $\Omega$ and by $\mathbb{G}_\alpha[\cdot]$ the associated \textit{Green operator} defined by
$$
  \mathbb{G}_\alpha[\nu](x) \defeq \int_\Omega G_\alpha(x, y)d\nu(y), \qquad \nu \in \mathcal{M}(\Omega),
$$
where $\mathcal{M}(\Omega)$ states the space of Radon measures on $\Omega$.

Throughout the manuscript, we shall assume the following on the data of problem \eqref{maineq}:

\begin{enumerate}
\item $\alpha \in \left(\frac{1}{2}, 1\right)$;
\item $0 \leq g \in C^0(\Omega \times [0, \infty))$ satisfying a growth condition of the form:
$$
g(x, s)\leq c s^{p}+\varepsilon|f(x)|,
$$
where $s, \varepsilon \geq 0$, $p \in (0, p^{\ast})$, $f \in L^{1}(\Omega)$ and $c$ is an appropriate positive constant;
\item $\varrho, \sigma \geq 0$ are constant;
\item $\nu \in \mathcal{M}(\Omega)$ is non-negative, $\mu \in \mathcal{M}(\mathbb{R}^{N}\setminus \Omega)$ with $\,\text{supp}(\mu) \subset \mathbb{R}^{N}\setminus \overline{\Omega}$ and $\mu(\mathbb{R}^{N}\setminus \Omega) < \infty$.
\end{enumerate}

In the following definition, we introduce the class of weak solutions  to problem \eqref{maineq} with homogeneous data in $\mathbb{R}^{N}\setminus \Omega$.

\begin{defi}\label{weak solu1} A function $u \in L^{1}(\Omega)$, with $|\nabla u|\in L_{loc}^{1}(\Omega)$ and  $g(x, |\nabla u|)\in L^{1}(\Omega)$, is a weak solution to problem
\begin{eqnarray}\label{boundary problem}
 \left\lbrace
  \begin{array}{l}
    (-\Delta)^{\alpha}u = g(x, |\nabla u|) + \sigma \nu  \quad \, \,\text{in } \Omega\\
    \qquad \quad  u=0  \,\,\,\,\,\,\,\quad\qquad \qquad\quad\text{ in } \mathbb{R}^{N}\setminus \Omega, \\
  \end{array}
  \right.
\end{eqnarray}
if for any $\phi \in \mathbb{X}_\alpha(\Omega)$, there holds
$$
   \int_{\Omega} u(-\Delta)^{\alpha}\phi dx = \int_\Omega\phi g(x, |\nabla u|)dx + \sigma \int_\Omega \phi d \nu.
$$
\end{defi}

For non-zero boundary data, we provide a definition based on the following: suppose that $u \in C^{2}(\mathbb{R}^{N})$, bounded and $\phi \in C^{2}(\mathbb{R}^{N})$ so that $\phi=0$ in $\mathbb{R}^{N}\setminus \overline{\Omega}$. If $u$ is a classical solution to \eqref{maineq}, we would have that $u$ is the probability density of the measure $\varrho \mu$ in $\mathbb{R}^{N}\setminus \Omega$ and moreover

\begin{equation}\label{nonlocal0}
\int_{\Omega} \phi (-\Delta)^{\alpha}u dx=\int_\Omega g(x, |\nabla u|)\phi dx + \sigma\int_\Omega \phi d \nu.
\end{equation}
By nonlocal integration by parts \cite{DOV}, it follows

\begin{equation}\label{nonlocal1}
\int_\Omega \phi (-\Delta)^{\alpha}u dx= \frac{C_{N, \alpha}}{2}\int_{\mathbb{R}^{2N}\setminus (\mathcal{C} \Omega)^{2}}\dfrac{(u(x)-u(y))(\phi(x)-\phi(y))}{|x-y|^{N+2\alpha}}dxdy -\int_{\mathbb{R}^{N}\setminus \Omega}\phi \mathcal{N}_\alpha u dx,
\end{equation}
where $\mathcal{N}_\alpha$ is the nonlocal normal derivative introduced in \cite{DOV}  given by
$$
   \mathcal{N}_\alpha v(x)= c_{N, \alpha}\int_\Omega \frac{v(x)-v(y)}{|x-y|^{N+2\alpha}}dy, \quad x \in \mathbb{R}^{N}\setminus \overline{\Omega},
$$
and
$$
\displaystyle  c_{N, \alpha} = \left(\int_{\Omega} \frac{dy}{|x-y|^{N+2\alpha}}\right)^{-1}
$$
is a universal constant. Since $\phi=0$ in $\mathbb{R}^{N}\setminus \overline{\Omega}$ and applying again integration by parts, equation \eqref{nonlocal1} becomes

\begin{equation}\label{nonlocal2}
\begin{split}
\int_\Omega \phi (-\Delta)^{\alpha}u dx & = \frac{C_{N, \alpha}}{2}\int_{\mathbb{R}^{2N}\setminus (\mathcal{C} \Omega)^{2}}\dfrac{(u(x)-u(y))(\phi(x)-\phi(y))}{|x-y|^{N+2\alpha}}dxdy \\ &  =\int_{\Omega} u (-\Delta)^{\alpha}\phi dx+  \int_{\mathbb{R}^{N}\setminus \Omega} u \mathcal{N}_\alpha \phi dx\\& = \int_\Omega u(-\Delta)^{\alpha}\phi dx + \varrho\int_{\mathbb{R}^{N}\setminus \Omega}\mathcal{N}_\alpha \phi d\mu.
\end{split}
\end{equation}
Observe that in view of the assumptions on $\mu$, the integral
$$
\int_{\mathbb{R}^{N}\setminus \Omega}\mathcal{N}_\alpha \phi d\mu
$$
is defined for all $\phi \in \mathbb{X}_\alpha(\Omega)$. Plugging \eqref{nonlocal2} into \eqref{nonlocal0}, we arrive at
\begin{equation}\label{nonlocal3}
\int_{\Omega} u (-\Delta)^{\alpha}\phi dx =\int_{\Omega} g(x, |\nabla u|)\phi dx+ \sigma\int_\Omega \phi d \nu - \varrho\int_{\mathbb{R}^{N}\setminus \Omega}\mathcal{N}_\alpha \phi d\mu,
\end{equation}
for any test function $\phi \in \mathbb{X}_\alpha(\Omega)$ (compare the expression \eqref{nonlocal3} with its local versions in \cite{VC1}, \cite{Marcus-V} and the references therein). Motivated by the above considerations we give the following definition:

\begin{defi}\label{weak solu2} A function $u \in L^{1}(\Omega)$, with $|\nabla u|\in L_{loc}^{1}(\Omega)$ and  $g(x, |\nabla u|)\in L^{1}(\Omega)$, is a weak solution to problem \eqref{maineq} if the integral equality \eqref{nonlocal3} holds for any $\phi \in \mathbb{X}_\alpha(\Omega)$.
\end{defi}

Regarding uniqueness, let us
present a useful comparison result to fractional quasi-linear
problems with gradient source:

\begin{teo}[{\bf Comparison Principle, \cite[Theorem 3.1]{AP}}]\label{CompPrin} Let $f \in L^1(\Omega)$ be a  non-negative  function.  Consider  $w_1,w_2$  two positive  functions such that $w_1, w_2\in W^{1, q}(\Omega)$ for all $q< p^{\ast}$, $(-\Delta)^\alpha w_1, (-\Delta)^\alpha w_2 \in L^1(\Omega)$ and,
$$
\left\{
\begin{array}{rclcl}
  (-\Delta)^\alpha w_1 & \leq  & G(x, \nabla w_1) & \text{in} &\Omega \\
  (-\Delta)^\alpha w_2 & \geq  & G(x, \nabla w_2) & \text{in} &\Omega \\
  w_1 & \leq & w_2 & \text{on} & \R^n \setminus \Omega,
\end{array}
\right.
$$where
$$
 \displaystyle   G: \Omega \times \R^N \to \R^{+} \quad \text{verifies} \quad |G(x, \nabla w_1)-G(x, \nabla w_2)|\leq C\mathrm{b}(x)|\nabla w_1-\nabla w_2|,
$$
for some  $C>0$ and  $\mathrm{b}\in L^{\sigma}(\Omega)$ with  $\sigma>\frac{N}{2\alpha-1}$. Then, $w_2 \geq w_1$ in $\Omega$.
\end{teo}

 A straightforward consequence of Theorem \ref{CompPrin}  is the following.

\begin{corol}\label{Uniqueness} Let $1 \leq p < p^{*}$ and  let $b$ and $C$ be as in Theorem \ref{CompPrin}.  Suppose that $g: \Omega \times \R_{+} \to \R_{+}$ satisfies:
\begin{equation}\label{assump g}
|g(x, |\nabla u_1|)-g(x, |\nabla u_2|)|\leq C\mathrm{b}(x)|\nabla u_1-\nabla u_2|,
\end{equation}for $u_1$ and $ u_2$ so that:
\begin{equation}\label{reg union}
u_1, u_2 \in \bigcup_{1 \leq q < p^{\ast}} W^{1, q}(\Omega)
\end{equation}and:
\begin{equation}\label{des u}
\left\{
\begin{array}{rclcl}
  (-\Delta)^\alpha u_1 & \leq  & g(x, |\nabla u_1|) + \sigma \nu& \text{in} &\Omega \\
  (-\Delta)^\alpha u_2 & \geq  & g(x, |\nabla u_2|) + \sigma \nu & \text{in} &\Omega \\
  u_1 & \leq & u_2 & \text{on} & \R^n \setminus \Omega,
\end{array}
\right.
\end{equation}Then $u_2 \geq u_1$.
\end{corol}
\begin{remark}
We point out that in Theorems \ref{superlinear p bounded} and \ref{MThm2}, we shall prove that solutions $u$ to \eqref{maineq} satisfy $\displaystyle u \in  W^{1, q}(\Omega)$ for all $1 \leq q < p^{*}$.
\end{remark}As an illustrative application of Corollary \ref{Uniqueness}, we may prove uniqueness for  \eqref{maineq} with:
\begin{equation}\label{g for uniq}
 g(x,|\nabla u|)= c(x)|\nabla u|^p+ \varepsilon|f(x)|, \qquad 0\leq c \in L^{\infty}(\Omega), \,\,f \in L^{1}(\Omega) \quad \text{and}\quad \varepsilon>0.
\end{equation}First, $g$ satisfies \eqref{assump g} for  $u_1$ and $u_2$ verifying \eqref{reg union} and \eqref{des u}. Indeed, since for all $\xi_1, \xi_2 \in \R^N$ and for all $p\geq 1$ we have
$$
\begin{array}{rcl}
  |\xi_1|^p-|\xi_2|^p & \leq & p|\xi_1|^{p-2}\langle \xi_1, \xi_1-\xi_2\rangle \\
   & \leq & p|\xi_1|^{p-1}|\xi_1-\xi_2|,
\end{array}
$$we get in the weak sense:
$$
\begin{array}{rcl}
  g(x,|\nabla u_1|)-g(x,|\nabla u_2|)  & \leq &  p\|c\|_{L^{\infty}(\Omega)}|\nabla u_1|^{p-1}|\nabla (u_1-u_2)|\\  &  = & b(x)|\nabla (u_1-u_2) |.
\end{array}
$$
Since $p<p^{\ast}, p^{\prime}>{p^{\ast}}^{\prime}$, $b$ falls
into the hypothesis of Theorem \ref{CompPrin} and $g$ satisfies \eqref{assump g}. Next, suppose that $u_1$ and $u_2$ are two solutions of \eqref{maineq} with $g$ as in \eqref{g for uniq} so that \eqref{reg union} holds.  Hence, by Corollary \ref{Uniqueness}, $u_1=u_2$.

\begin{remark}It is worthwhile mentioning that, according to our knowledge, uniqueness of solutions  for $p\in (0, 1)$, general sources and $g$ as an absorption or source term  is an  open problem, even in the local case driven by the Laplacian operator (see comments after Theorem 1.1 in \cite{NV} and in \cite{CV}).
\end{remark}

Regarding existence, in Section \ref{subcritical} we shall prove that there
exist weak solutions to \eqref{maineq} under polynomial
growth conditions on the non-linearity $g$ and having the
following representation formula:
\begin{equation}\label{decomposition9}
u(x)=\mathbb{G}_\alpha[g(x, |\nabla u|)+ \sigma \nu](x)+\varrho \mathbb{P}_\alpha[\mu](x)
\end{equation}
where for $x \in \Omega$:
$$
   \mathbb{G}_\alpha[g(x, |\nabla u|)+ \sigma \nu](x)=\int_\Omega g(x, |\nabla u(y)|)G_\alpha(x, y)dy+\sigma\int_\Omega G_\alpha(x, y)d\nu(y)
$$
represents the ``Green potential'' and
$$
   \mathbb{P}_\alpha[\mu](x)\defeq -\int_{\mathbb{R}^{N}\setminus \Omega} (\mathcal{N}_\alpha G_\alpha(x, \cdot))(y)d\mu(y)
$$
represents the ``Poisson potential'' (compare with \cite[Section 2.4]{Ve4} in the local scenario).

Finally, we must point out that the expression $\varrho\mathbb{P}_{\alpha}[\mu]$ in the decomposition \eqref{decomposition9} plays the role of Poisson operator in the nonlocal setting. In this regard, it is interesting to compare the  expression \eqref{decomposition9} with related results such as \cite[Proposition 1.1.3 and Theorem 1.2.2]{Marcus-V} and \cite[Proposition A]{NV}.

Continuing, let us discuss the expression \eqref{decomposition9}. Formally, the function $v=u-\varrho \mathbb{P}_\alpha[\mu]$ in \eqref{decomposition9} satisfies
\begin{equation}\label{diri}
\begin{cases}
(-\Delta)^\alpha v=g(x, |\nabla( v+\varrho\mathbb{P}_{\alpha}[\mu])|) + \sigma \nu &  \text{in } \Omega,\\
\qquad \quad v=0 &  \text{in } \R^N\setminus\Omega
\end{cases}
\end{equation}
in the sense of Definition \ref{weak solu1}. The existence of $v$
will be the main topic of Section \ref{subcritical}. On the other
hand, $w=\mathbb{P}_\alpha[\mu]$ solves

\begin{equation}\label{poisson}
\begin{cases}
(-\Delta)^\alpha w=0 &  \text{in } \Omega,\\
\qquad \quad w=\mu &  \text{in } \R^N\setminus\Omega
\end{cases}
\end{equation}
in the sense of Definition \eqref{weak solu2}. To see this, we introduce the auxiliary function:
$$
w_\mu(x)= C_{N, \alpha}\int_{\R^N \setminus\Omega}
\frac{1}{|z-x|^{N+2\alpha}}\,d\mu(z), \qquad x \in \Omega
$$
so that
\begin{equation}\label{omega-mu}
\int_\Omega w_\mu \phi\, dx = -\int_{\mathbb{R}^{N}\setminus
\Omega}\mathcal{N}_\alpha \phi d\mu, \quad \textnormal{ for any
$\phi \in \mathbb{X}_\alpha(\Omega)$},
\end{equation}
and hence
\begin{equation}\label{111}
\mathbb{P}_\alpha[\mu](x)=\int_\Omega w_\mu(z)G_\alpha(x, z)dz=\mathbb{G}_\alpha[w_\mu](x).
\end{equation}
In view of our assumptions on $\mu$ and \cite[Lemma 5.2]{VC1}, it holds $w_\mu \in C^{1}(\overline{\Omega})$. Consequently, by \cite[Proposition 2.4]{CV} and \eqref{111}, we have
 \begin{equation}\label{REG p}
 \mathbb{P}_\alpha[\mu] \in \bigcup_{1 \leq q < p^{\ast}}W_0^{1, q}(\Omega).
 \end{equation}Moreover, since
\begin{equation}\label{property omega}
\int_\Omega \mathbb{P}_\alpha[\mu] (-\Delta)^{\alpha} \phi dx = \int_\Omega \phi w_\mu dx = -\int_{\mathbb{R}^{N}\setminus \Omega}\mathcal{N}_\alpha \phi d\mu \quad \textnormal{for all $\phi \in \mathbb{X}_\alpha(\Omega)$}
\end{equation}we have that   $\mathbb{P}_\alpha[\mu]$ solves \eqref{poisson} in the sense of Definition \eqref{weak solu2}. Consequently, assuming for the moment that $v$ solves \eqref{diri}, the function $u$ as in \eqref{decomposition9} is indeed a solution of problem \eqref{maineq}:
\begin{equation*}
\begin{split}
\int_\Omega u (-\Delta)^{\alpha}\phi dx &= \int_\Omega v  (-\Delta)^{\alpha}\phi dx + \varrho \int_\Omega \mathbb{P}_\alpha[w_\mu] (-\Delta)^{\alpha}\phi dx \\ & =  \int_\Omega g(x, |\nabla u|)\phi dx + \sigma\int_\Omega \phi d\nu + \varrho \int_\Omega w_\mu \phi dx\\ & =   \int_\Omega g(x, |\nabla u|)\phi dx+ \sigma\int_\Omega \phi d\nu - \varrho \int_{\mathbb{R}^{N}\setminus \Omega}\mathcal{N}_\alpha \phi d\mu, \quad \textnormal{for all $\phi \in \mathbb{X}_\alpha(\Omega)$,}
\end{split}
\end{equation*}
where we have used \eqref{111} and \eqref{omega-mu} in the latter two equalities. Therefore, it remains to prove that problem \eqref{diri} admits a solution.

\section{Main results: Sub-critical case $0< p < \frac{N}{N-(2\alpha-1)}$}\label{subcritical}

In this section, we prove existence of weak solution to the elliptic integral-differential problem \eqref{maineq} under a $p-$polynomial growth condition on $g$.  As it was discussed in the previous section, it is enough to solve problem \eqref{diri}.

We assume throughout the section that:
$$
  p<p^{\ast}\defeq \frac{N}{N-(2\alpha-1)}.
$$
Furthermore, we will divide the exposition in two sub-cases:
\begin{enumerate}
  \item Super-linear case, i.e., $p>1$;
  \item Sub-linear case, i.e., $p\leq 1$.
\end{enumerate}

\subsection{Super-linear case}

Firstly, we will treat the super-linear setting. We provide existence result under appropriated growth/integrability conditions on the gradient source term. Moreover, such solutions fulfils a certain explicit characterization.

\begin{teo}\label{superlinear p bounded}
Suppose that $g$ satisfies the following growth hypothesis:
\begin{equation}\label{H g}
 g(x, s)\leq c s^{p}+\varepsilon|f(x)|, \quad s \geq 0,\,\varepsilon \geq 0,
\end{equation}
where $f \in L^{1}(\Omega)$  and  $1\leq p < p^{\ast}$. Then,  for all small enough $c$,  problem \eqref{maineq} admits a non-negative weak solution $u$  which fulfils the decomposition \eqref{decomposition9} and satisfies:
$$u \in \bigcup_{1 \leq q < p^{\ast}} W^{1, q}(\Omega).$$
\end{teo}

\begin{proof}[Proof of Theorem \ref{superlinear p bounded}]
Firstly, we will approximate the nonlinearity $g$ and Radon measure $\nu$ by regular sequences $\{g_n\}_{n \in \mathbb{N}}$ and $\{\nu_n\}_{n \in \mathbb{N}}$ respectively. For that end, consider sequences of non-negative functions $\nu_n \in C^{1}(\overline{\Omega})$ and $ g_n\in C^{1}(\Omega, [0, +\infty))$  such that
\begin{equation}\label{conv measures nu}
\lim_{n\to \infty}\int_{\overline{\Omega}}\xi \nu_n dx=\int_{\overline{\Omega}}\xi d\nu \qquad \textnormal{for all $\xi \in C^0(\overline{\Omega})$}
\end{equation}
and
\begin{enumerate}
\item $g_n(x, 0)=g(x, 0)$ for every $x \in \Omega$;
\item $g_n\leq g_{n+1}\leq g$ and $\sup g_n(x, s)=n$;
\item $\|g_n-g\|_{L^\infty_{loc}(\Omega \times \R_+)} \to 0$ as $n \to \infty$.
\end{enumerate}
By \eqref{conv measures nu}, we have for all $n \gg 1$ (large enough)
\begin{equation}\label{uniform bd of nu_n}
\sup_{n}\|\nu_n\|_{L^{1}(\Omega)}=\sup_{n}\int_{\overline{\Omega}}d\nu_n  \leq C_0
\end{equation}
where  $C_0 \defeq  ||\nu ||_{\mathcal{M}(\Omega)}+1.$ To solve \eqref{diri}, we first find approximations by solving the problems
\begin{equation}\label{dirign}
\begin{cases}
(-\Delta)^\alpha u=g_n(x, |\nabla( u+\varrho\mathbb{P}_{\alpha}[\mu])|) +\sigma \nu_n&  \text{in } \Omega,\\
\qquad \quad u=0 &  \text{in } \R^n\setminus\Omega;
\end{cases}
\end{equation}
By fixed-point methods, we shall prove that \eqref{dirign} admits a non-negative solution $v_n$ such that
$$
   \|\nabla v_n\|_{L^{p}(\Omega)}\leq\lambda^{\ast}
$$
 uniformly for some $\lambda^{\ast}>0$ (to be determined \textit{a posteriori}). For this purpose, we define the closed, convex and bounded sets
$$
  \mathcal{G}_{\lambda}=\{u\in
W_0^{1,p}(\Omega): \|\nabla u\|_{L^{p}(\Omega)}\leq\lambda\}
$$
and operators $T_n$ on $\mathcal{G}_{\lambda}$ as follows: for each $v \in \mathcal{G}_{\lambda}$, let $v_n=T_n(v)$ be the weak solution to
\begin{equation}\label{dirign1}
\begin{cases}
(-\Delta)^\alpha w=g_n(x, |\nabla( v+\varrho\mathbb{P}_{\alpha}[\mu])|) +\sigma \nu_n&  \text{in } \Omega,\\
\qquad \quad  w=0 &  \text{in } \R^n\setminus\Omega.
\end{cases}
\end{equation}Observe that the following representation holds:
$$v_n(x)=\int_\Omega G_\alpha(y, x)[g_n(y, |\nabla( v+\varrho\mathbb{P}_{\alpha}[\mu])|) +\sigma \nu_n]\,dy.$$ See for instance \cite[pag. 7]{AP20}.  Moreover, $v_n \in \mathcal{C}^{s}(\mathbb{R}^{n}) \cap W^{1, q}(\Omega)$ for all $q \in [1, p^{\ast})$ by \cite[Proposition 1.1]{RS}  and \cite[Proposition 2.3]{CV}.  Hence, $v \in W_0^{1, q}$ for all $q \in [1, p^{\ast})$.

We will check that $T_n(\mathcal{G}_{\lambda^{\ast}}) \subset \mathcal{G}_{\lambda^{\ast}}$ for all $n \geq 1$. Recalling  \cite[Proposition 2.4]{CV} there exists  $c_0 >0$ so that
\begin{equation}\label{main eq}
\begin{split}
||\nabla T_n (v) \|_{L^{p}(\Omega)}&\leq c_0\|g(x, |\nabla(
v+\varrho\mathbb{P}_{\alpha}[\mu])|)+\sigma\nu_n\|_{L^1(\Omega)}\\
&\leq c_0\left(c\int_{\Omega}|\nabla(
v+\varrho\mathbb{P}_{\alpha}[\mu])|^p dx+\varepsilon \|f\|_{L^{1}(\Omega)} +\sigma C_0\right)\\
&\leq c_0\left(c2^{p-1}{\lambda^{\ast}}^{p}+
c2^{p-1}\varrho^{p}||\nabla(\mathbb{P}_{\alpha}[\mu])||_{L^{p}(\Omega)}^p+\varepsilon \|f\|_{L^{1}(\Omega)} +\sigma C_0 \right).
\end{split}
\end{equation}

Let us consider the auxiliary function:
$$
  \mathrm{F}(\lambda)=c_0\left(c2^{p-1}\lambda^{p-1}+
\frac{c2^{p-1}\varrho^{p}||\nabla(\mathbb{P}_{\alpha}[\mu])||_{L^{p}(\Omega)}^p+\varepsilon \|f\|_{L^{1}(\Omega)}  +\sigma C_0  }{\lambda}\right)-1.
$$
Now, choose $\lambda >0$ such that
\begin{equation}\label{ineqq}
\frac{\varepsilon\|f\|_{L^{1}(\Omega)}  +\sigma C_0 }{\lambda}=\frac{1}{2c_0},
\end{equation}
and take $c>0$ in \eqref{H g} small enough such that
\begin{equation}\label{ineqqq}
c2^{p-1}\left(\lambda^{p-1}+\frac{\varrho^{p}||\nabla(\mathbb{P}_{\alpha}[\mu])||_{L^{p}(\Omega)}^p}{\lambda}\right) <\frac{1}{2c_0}.
\end{equation}
Hence, $\mathrm{F}(\lambda)<0$ for such a $\lambda$. Also, observe that $\mathrm{F}(\lambda)>0$ for small enough $\lambda$. Thus, there exists $\lambda^{\ast}>0$ so that $\mathrm{F}(\lambda^{\ast})=0$. Choosing $\lambda^{\ast}$ in $\mathcal{G}_{\lambda}$, the inequalities \eqref{main eq} imply that $T_n(v)\in \mathcal{G}_{\lambda^{\ast}}$ for all $n$. This shows that $T_n$ maps $\mathcal{G}_{\lambda^{\ast}}$ into itself.

 Clearly, if $u_k \to u$ in $W^{1, p}(\Omega)$ as $k \to \infty$, then $g_n(x, \nabla (u_k+\rho\mathbb{P}_{\alpha}[\mu])) \to g_n(x, \nabla (u+\rho\mathbb{P}_{\alpha}[\mu]))$ in $L^1(\Omega)$ as $n \to \infty$. Hence, $T_n$ is a continuous map. We prove now that $T_n$ is a compact operator. For each $n$, let $v_{n, k}$ be a sequence in $\mathcal{G}_{\lambda^{\ast}}$. By definition of $\mathcal{G}_{\lambda^{\ast}}$ and Poincar\'e inequality, $v_{n, k}$ is bounded in $k$ in $W_0^{1, p}(\Omega)$. Observe that
 $$
   g_n(x, \nabla( v_{n, k}+\varrho\mathbb{P}_\alpha[\mu]))+\sigma \nu_n
 $$
 is bounded (in $k$) in $L^{1}(\Omega)$ and hence by the compactness of the operator
 $$
    L^{1}(\Omega)\ni h \to \nabla \mathbb{G}_\alpha[h]\in L^{p}(\Omega),
 $$
 there is a converging subsequence of $\nabla T_n(v_{n, k})$. We conclude that $T_n(v_{n, k})$ admits a converging subsequence in $W_0^{1, p}(\Omega)$. From Schauder's fixed point Theorem,  there exists $v_n\in \mathcal{G}_{\lambda^{\ast}}$ such that $T_n(v_n)=v_n$ and
$\|\nabla v_n\|_{L^{p}(\Omega)}\leq\lambda^{\ast}$. It remains to prove that $v_n \to v$, where $v$ solves \eqref{diri}.

Since $v_n=T(v_n)$, we have

\begin{equation}\label{bf lim}
\begin{split}\int_{\Omega}v_n (-\Delta)^\alpha\psi\,dx&=\int_\Omega g_n(x, |\nabla(
v_n+\varrho\mathbb{P}_{\alpha}[\mu])|)\psi\,dx+\sigma \int_{\Omega}\psi d\nu_n, \,\,\mbox{ for all
}\psi\in\mathbb{X}_{\alpha}(\Omega).\\&
\end{split}
\end{equation}
Hence,
\begin{equation}\label{u sub n}
v_n=\mathbb{G}_\alpha[g_n(x, |\nabla (v_n+\varrho\mathbb{P}_{\alpha}[\mu]) |)+\sigma \nu_n]
\quad \textnormal{ in } \Omega.
\end{equation}
Moreover, assumption \eqref{H g} and $$\|\nabla (v_n+\varrho\mathbb{P}_{\alpha}[\mu]) \|_{L^{p}(\Omega)}\leq\lambda^{\ast} + \varrho \|\nabla\mathbb{P}_\alpha[\mu]\|_{L^{p}(\Omega)}$$yield that   $g_n(x, |\nabla(v_n+\varrho\mathbb{P}_{\alpha}[\mu])|)$ is uniformly bounded in $L^{1}(\Omega)$. From \eqref{u sub n} and  the compactness of the operator $$L^{1}(\Omega) \ni h \to (\mathbb{G}_\alpha[h], \nabla \mathbb{G}_\alpha[h])\in L^{p}(\Omega)\times L^{p}(\Omega)$$it follows the strong convergence, up to subsequence, of $v_n$ in  $W^{1, p}_0(\Omega)$ to some $v \in W_0^{1, p}(\Omega)$. Thus $\nabla v_n$ converges point-wisely to $\nabla v$, and  so
$$
   g_n(x, |\nabla (v_n+\varrho \mathbb{P}_\alpha[\mu])|) \to g(x, |\nabla (v+\varrho \mathbb{P}_\alpha[\mu])|) \quad \text{a.e. in} \quad \Omega.
$$

In the sequel, we prove that $g_n(x, |\nabla (v_n+\varrho \mathbb{P}_\alpha[\mu])|)$ is uniformly integrable. For that end, observe that for any Borel subset $E\subset \Omega$
\begin{equation}\label{eq 12}
\begin{split}
\int_E g_n(x, |\nabla (v_n+\varrho \mathbb{P}_\alpha[\mu])|)dx &\leq
c \int_E |\nabla (v_n+\varrho \mathbb{P}_\alpha[\mu])|^{p}dx +
\varepsilon\|f\|_{L^{1}(E)} \\& \leq c2^{p-1}||\nabla v_n-\nabla
v||_{L^{p}(E)}^{p}  + c2^{p-1}||\nabla( v+\varrho
\mathbb{P}_\alpha[\mu]) \|_{L^{p}(E)}^{p} \\&+ \varepsilon\|f\|_{L^{1}(E)}.
\end{split}
\end{equation}
Let $\hat{\eta} > 0$ be arbitrary. Then there are $N_0, \delta_0 > 0$ so that $n\geq N_0$ implies
\begin{equation}\label{eq 123}
c2^{p-1}||\nabla v_n-\nabla v||_{L^{p}(\Omega)}^{p} < \frac{\hat{\eta}}{3},
\end{equation}
and $|E|< \delta_0$ gives
\begin{equation}\label{eq 1234}
\max\left\{\epsilon \|f\|_{L^{1}(E)},\,\, c2^{p-1}||\nabla (v+\varrho \mathbb{P}_\alpha[\mu]) \|_{L^{p}(E)}^{p}\right\} <\frac{\hat{\eta}}{3}.
\end{equation}
In addition, for each $n \in \left\lbrace 1, ..., N_0\right\rbrace$, there is $\delta_n > 0$ so that $|E|<\delta_n$ implies:
\begin{equation}\label{eq 12345}
c2^{p-1}\|\nabla v_n-\nabla v||_{L^{p}(E)}^{p}  < \frac{\hat{\eta}}{3}.
\end{equation}
Choose  $\delta \defeq \min\left\lbrace \delta_0, \delta_1, ..., \delta_{N_0} \right\rbrace$ and $|E|< \delta$. Hence, plugging \eqref{eq 123}, \eqref{eq 1234} and \eqref{eq 12345} into \eqref{eq 12} gives:
$$
  \int_E g_n(x, |\nabla (v_n+\varrho \mathbb{P}_\alpha[\mu])|)dx < \hat{\eta} \quad \textnormal{ for all n}.
$$
Thus $g_n(x, |\nabla (v_n+\varrho \mathbb{P}_\alpha[\mu])|)$ is uniformly integrable. By Vitali's convergence theorem we obtain
$$
  g_n(x, |\nabla (v_n+\varrho \mathbb{P}_\alpha[\mu])|) \to g(x, |\nabla (v+\varrho \mathbb{P}_\alpha[\mu])|) \quad\textnormal{ in }L^{1}(\Omega).
$$
Therefore, taking $n \to \infty$ in \eqref{bf lim} it holds:
$$
   \int_{\Omega}v(-\Delta)^\alpha\psi\,dx=\int_\Omega g(x, |\nabla
(v+\varrho \mathbb{P}_\alpha[\mu])|)\psi\,dx + \sigma\int_\Omega \psi d\nu.
$$
Hence, $v$ is a weak solution of problem \eqref{diri}. Finally, by writing
$$
   u=v+\varrho \mathbb{P}_\alpha[\mu]
$$
we obtain a solution to \eqref{maineq}.
\end{proof}

\begin{remark}Observe that to derive the existence of a positive root for $\mathrm{F}$ in the above proof, one may ask for $\varrho$, $\epsilon$ and $\sigma$ to be sufficiently small instead of the imposed condition on the size of $c$.
\end{remark}

\subsection{Sub-linear case}\label{SecSub-Linear}

In the sequel, we will deal with the sub-linear scenario. Similarly to the previous section, we provide existence of weak solutions under appropriated growth/integrability conditions on the gradient source term.

\begin{teo}\label{MThm2}
Suppose that $g$ satisfies:
\begin{equation}\label{H g1}
 g(x, s)\leq c s^{p}+\varepsilon |f(x)|, \quad  x\in \Omega, s \geq 0\,\,\epsilon \geq 0,
\end{equation}where $f$ is integrable and  $0\leq p\leq 1$. Then,  for all small enough $c$,  problem \eqref{maineq} admits a non-negative weak solution $u$  which fulfils the decomposition \eqref{decomposition9} and satisfies:
$$u \in \bigcup_{1 \leq q < p^{\ast}} W^{1, q}(\Omega).$$
\end{teo}

\begin{proof}
We proceed as in the proof of Theorem \ref{superlinear p bounded}. We point out the differences. Consider the sequences $\{g_n\}_{n \in \mathbb{N}}$ and $\{\nu_n\}_{n \in \mathbb{N}}$ as in the super-linear case, and define the operators
$$
  T_n(v)=\mathbb{G}_\alpha[g_n(x, |\nabla( v+\varrho\mathbb{P}_{\alpha}[\mu])|)+\sigma \nu_n]
$$
 for $v$ in the set
$$
\mathcal{O}_{\bar{\lambda}} \defeq \left\lbrace v \in W^{1, 1}_0(\Omega): \|\nabla v\|_{L^1{(\Omega)}}\leq \bar{\lambda} \right\rbrace
$$
for some $\bar{\lambda}>0$ (to be adjusted \textit{a posteriori}). First of all, we show that $T_n$ maps $\mathcal{O}_{\bar{\lambda}}$ into itself. Observe that  \cite[Proposition 2.4]{CV}  and \eqref{H g} yield

\begin{equation}\label{long calcu}
\begin{split}
&\|\nabla T_n(v)\|_{L^{1}(\Omega)} \leq c_0\int_\Omega\left(
c|\nabla (v+\varrho\mathbb{G}_\alpha[w_\mu])|^{p}+\epsilon|f(x)| + \sigma
\nu_n\right)dx\\&\quad \leq c_0\left(\int_{\Sigma} c|\nabla
(v+\varrho\mathbb{P}_\alpha[\mu])|^{p}dx + \int_{(\Sigma)^c} c|\nabla
(v+\varrho\mathbb{P}_\alpha[\mu])|^{p}dx +\sigma\int_\Omega \nu_n
dx +\epsilon\|f\|_{L^{1}(\Omega)}\right) \\& \quad \leq c_0.c.\left( \|\nabla
v\|_{L^{1}(\Omega)} + \varrho\|\nabla
\mathbb{P}_\alpha[\mu]\|_{L^{1}(\Omega)}+|\Omega|+ \frac{\epsilon}{c}\|f\|_{L^{1}(\Omega)}\right)
+c_0.\sigma.\left(\sup_{n}\nu_n(\overline{\Omega})\right)\\ & \quad
\leq c_0.c.\left( \bar{\lambda} +\varrho\|\nabla
\mathbb{P}_\alpha[\mu]\|_{L^{1}(\Omega)}+\frac{\epsilon}{c}\|f\|_{L^{1}(\Omega)}+|\Omega|\right) +c_0.\sigma.\left(\sup_{n}\nu_n(\overline{\Omega})
\right),
\end{split}
\end{equation}
where $\Sigma: = \left\lbrace x \in \Omega: |\nabla (v+\varrho\mathbb{P}_\alpha[\mu])|>1\right\rbrace$. Now, consider:
\begin{equation*}
\mathbb{F}(\lambda)=c_0.c.\left(1+\frac{\varrho\|\nabla \mathbb{P}_\alpha[\mu]\|_{L^{1}(\Omega)}+\epsilon\|f\|_{L^{1}(\Omega)}/c+|\Omega|}{\lambda}\right) +\frac{c_0.\sigma.\left(\sup_{n}\nu_n(\overline{\Omega}) \right) }{\lambda} -1.
\end{equation*}
Choose $c < \frac{1}{c_0}$. Hence for $\lambda \gg 1$ large enough it holds $\mathbb{F}(\lambda)< 0$. Moreover, fixing $c$ as before, we have that  $\mathbb{F}(\lambda)>0$ for $\lambda$ sufficiently closed to $0$. Hence, there is $\bar{\lambda} > 0$ so that $\mathbb{F}(\bar{\lambda})=0$. We take $\bar{\lambda}$ in $\mathcal{O}_{\bar{\lambda}}$ and thus $T_n: \mathcal{O}_{\bar{\lambda}} \to \mathcal{O}_{\bar{\lambda}}$. Moreover, $T_n$ is continuous and compact. Therefore, for each $n$ there is $v_n \in \mathcal{O}_{\bar{\lambda}}$ so that $T_n(v_n)=v_n$ and $||\nabla v_n||_{L^{1}(\Omega)}\leq \bar{\lambda}$. Observe that $g_n(x, |\nabla (v_n+\varrho \mathbb{P}_\alpha[\mu])|)$ is bounded in $L^{1}(\Omega)$. In effect,
\begin{equation*}
\begin{split}
\int_\Omega g_n(x,|\nabla (v_n+\varrho \mathbb{P}_\alpha[\mu])|)dx
&\leq c\left( \int_{\Sigma^n} |\nabla (v_n+\varrho
\mathbb{P}_\alpha[\mu])|^{p} dx + \int_{
(\Sigma^n)^c} |\nabla (v_n+\varrho
\mathbb{P}_\alpha[\mu])|^{p} dx\right)\\ &+ \epsilon \|f\|_ {L^{1}(\Omega)}\\
&  \leq c
\left(\bar{\lambda}+\varrho\|\nabla\mathbb{P}_\alpha[\mu]\|_{L^{1}(\Omega)}\right)+
c|\Omega|+\epsilon\|f\|_{L^{1}(\Omega)},
\end{split}
\end{equation*}
where $\Sigma^n = \left\lbrace x \in \Omega: |\nabla (v_n+\varrho\mathbb{P}_\alpha[\mu])|>1\right\rbrace$. Hence, for all $q \in [1, p^{\ast})$,
$$
  v_n=\mathbb{G}_\alpha[g_n(x, |\nabla (v_n+\varrho \mathbb{P}_\alpha[\mu])|+\sigma\nu_n)]
$$
has a converging subsequence in $W_0^{1, q}(\Omega)$. In particular, there exists $v$ so that $$v \in \bigcup_{1 \leq q < p^{\ast}}W_0^{1, q}(\Omega)$$ and
$$
  g_n(x, |\nabla (v_n+\varrho \mathbb{P}_\alpha[\mu])|) \to g(x, |\nabla (v+\varrho \mathbb{P}_\alpha[\mu])|), \quad \textnormal{a.e. in } \Omega.
$$
To show that $g_n(x, |\nabla (v_n+\varrho \mathbb{P}_\alpha[\mu])|)$ is uniformly integrable we proceed as in the proof of Theorem \ref{superlinear p bounded}. We write \eqref{eq 12} as follows:
\begin{equation*}
\begin{split}
&\int_E g_n(x, |\nabla (v_n+\varrho \mathbb{P}_\alpha[\mu])|)dx \leq \\ &
\qquad \qquad \leq c\left( \int_{\Sigma_E^n} |\nabla (v_n+\varrho
\mathbb{P}_\alpha[\mu])|^{p}dx + \int_{(\Sigma^n_{E})^c} \nabla (v_n+\varrho
\mathbb{P}_\alpha[\mu])|^{p}dx\right) +\epsilon\|f\|_{L^{1}(E)}\\ & \qquad
\qquad \leq c ||\nabla (v_n+\varrho
\mathbb{P}_\alpha[\mu])||_{L^{1}(E)}+c|E|+\epsilon\|f\|_{L^{1}(E)}\\&
\qquad \qquad \leq c \left( ||\nabla v_n-\nabla v||_{L^{1}(E)}+ ||\nabla
(v+\varrho
\mathbb{P}_\alpha[\mu])||_{L^{1}(E)}\right)+c|E|+\epsilon\|f\|_{L^{1}(E)},
\end{split}
\end{equation*}
where $\Sigma_E^n = E \cap \left\lbrace |\nabla (v_n+\varrho
\mathbb{P}_\alpha[\mu])|>1\right\rbrace$. The rest of the proof is the same as in Theorem \ref{superlinear p bounded}.
\end{proof}

\section{Stability results}\label{Stability}

In this section, we prove the stability of solutions to \eqref{maineq} under appropriate perturbations of the measures involved.

\begin{teo}\label{stability}
Assume $1 \leq  p < p^{\ast}$. Consider the problem
\begin{equation}\label{estability}
\begin{cases}
(-\Delta)^\alpha w=g(x, |\nabla w|)+\sigma \nu &  \text{in } \Omega,\\
  \qquad \quad w=\varrho\mu &  \text{in } \R^N\setminus\Omega
\end{cases}
\end{equation}where $g$ satisfies
\begin{equation}\label{asump g
all p} 0 \leq g(x, s)\leq c s^{p}+\epsilon|f(x)|, \quad s, \epsilon
\geq 0, f\in L^{1}(\Omega)
\end{equation}and \eqref{assump g}. Suppose that
$$
\mu_n\rightharpoonup\mu \quad \textnormal{\textit{ in the sense of
duality in} }  C_0(\R^N\setminus\Omega),
$$
and
$$
\nu_n \to \nu \quad \textnormal{\textit{ in the sense of duality in } } C(\overline{\Omega}), \quad \nu_n \in L^{1}(\Omega).
$$
If  $u_n$
is the weak solution of
\begin{equation}\label{maineqsta}
\begin{cases}
(-\Delta)^\alpha v=g(x, |\nabla v|)+\sigma \nu_n &  \text{in } \Omega,\\
  \qquad \quad v=\varrho\mu_n &  \text{in } \R^N\setminus\Omega,
\end{cases}
\end{equation}
then there is $u$ so that
\begin{equation}\label{conv u limit}
 u_n\to u \textnormal{ \textit{strongly in} } W^{1,p}(\Omega), \textnormal{ if $1\leq p < p^{\ast}.$}
\end{equation}Moreover, the limiting profile $u$ solves \eqref{estability}.

\end{teo}

\begin{proof}
 Let $u_{n_k}$ be a subsequence of $u_n$. By  Theorems \ref{CompPrin} and  \ref{superlinear p bounded} it follows that
$$
  u_{n_k}=\mathbb{G}_\alpha[g(x, |\nabla u_{n_k}|)+ \sigma \nu_{n_k}]+\varrho\mathbb{P}_\alpha[\mu_{n_k}].
$$
By \eqref{property omega} we have for any $\phi \in \mathbb{X}_\alpha(\Omega)$
$$
   \int_\Omega \mathbb{P}_\alpha[\mu_{n_k}](-\Delta )^{\alpha}\phi\,dx = \int_\Omega w_{\mu_{n_k}}\phi\,dx \to \int_\Omega w_{\mu}\phi\,dx =\int_\Omega \mathbb{P}_\alpha[\mu](-\Delta )^{\alpha}\phi\,dx  \textnormal{ as }k \to \infty,
$$
by the assumptions on $\mu$ stated in Section \ref{prelim} and  dominated convergence theorem. We conclude that $\mathbb{P}_\alpha[\mu]$ solves problem \eqref{poisson}.

On the other hand, $v_{n_k} = \mathbb{G}_\alpha[g(x, |\nabla u_{n_k}|)+ \sigma \nu_{n_k}]$  satisfies
\begin{equation}\label{eqqqq}
\int_\Omega v_{n_k}(-\Delta)^{\alpha}\phi\,dx=\int_\Omega g(x, |\nabla
u_{n_k}|) \phi\,dx + \sigma\int_\Omega \phi d\nu_{n_k}. \quad \textnormal{
for all }  k.
\end{equation}
By assumption, $$\int_\Omega \phi d\nu_{n_k} \to
\int_\Omega \phi d\nu$$for all test function $\phi$. Appealing to
the proof of Theorem \ref{superlinear p bounded}, we have that there
is $\lambda^{\ast}_{n_k}>0$ so that
$$
   \|\nabla v_{n_k}\|_{L^{p}(\Omega)}\leq \lambda^{\ast}_{n_k}.
$$
Observe that it is possible to choose $\lambda^{\ast}$ so that $\lambda^{\ast}_{n_k} \leq \lambda^{\ast}$ for all $k$. This follows from the uniform boundedness of $\nu_{n_k}$ in $L^{1}(\Omega)$,  \eqref{ineqq}, \eqref{ineqqq}, and the fact that
$$
   \|\nabla\mathbb{P}_\alpha[\mu_{n_k}]\|_{L^{p}(\Omega)} \leq C_0 \|w_{\mu_{n_k}}\|_{L^{1}(\Omega)}\leq C,\quad \textnormal{for all }n.
$$
Hence, $g(x, |\nabla u_{n_k}|)+\sigma\nu_{n_k}$ is uniformly bounded in $L^{1}(\Omega)$. From the compactness of $\nabla \mathbb{G}_\alpha$, there is a further subsequence, $v_{n_{k_i}} $ converging to some $v$ in $W^{1, p}(\Omega)$. Moreover, since $\mathbb{P}_\alpha[\mu_{n}]=\mathbb{G}_\alpha[w_{\mu_{n}}]$ for all $n$ and $w_{\mu_{n}}$ is uniformly bounded in $L^{1}(\Omega)$, it follows that
$$
  u_{n_{k_i}} \to u \defeq  v+\varrho \mathbb{P}_\alpha[\mu]
$$
strongly in $W^{1, p}(\Omega)$. Appealing to Vitali's converging theorem and taking $i \to \infty$ in \eqref{eqqqq}, it follows that $v$ solves \eqref{diri}. Therefore, $u$ solves \eqref{estability}. Hence, every subsequence of $u_n$ has a further subsequence converging, by uniqueness,  to the same  limit $u$. Therefore, the whole sequence $u_n$ converges to $u$.
\end{proof}

\begin{remark}
In the cases $p \in [1, \infty)$ with general $g$ or $p \in (0, 1)$,  the statement \eqref{conv u limit} holds true for a subsequence due to a possible lack of uniqueness.
\end{remark}

\section{Existence results for measures concentrated on the boundary}\label{Sec6}

In the previous sections, we have considered integro-differential problems  with boundary values $\mu$ supported in  $\mathbb{R}^{N}\setminus \overline{\Omega}$.  In this part we shall add boundary measures $\eta$ concentrated on $\partial \Omega$. Precisely, we shall be interested in studying existence of solutions to the following problems

 \begin{equation}\label{maineq-conc}
\left\{
\begin{array}{rcll}
  (-\Delta)^\alpha u & = &g(x, |\nabla u|) +\sigma \nu &  \text{in } \Omega,\\
  u&=& \eta & \text{on }\partial \Omega,\\
  u& = & \varrho\mu &  \text{on } \R^N\setminus\overline{\Omega}.
\end{array}
\right.
\end{equation}

For a given Radon measure $\eta \in  \mathcal{M}(\overline{\Omega})$  supported on $\partial \Omega$, we first consider the simpler problem

\begin{equation}\label{concentrated1}
\begin{cases}
(-\Delta)^\alpha u=0 &  \text{in } \Omega,\\
  \qquad \quad u=\eta &  \text{on } \partial \Omega.
\end{cases}
\end{equation}

In the local case (see for instance \cite{BVV}, \cite{GV} and  \cite[Chapter 1]{Marcus-V}) a solution $u \in L^{1}(\Omega)$ to the problem
\begin{equation*}
\begin{cases}
-\Delta u=\mu &  \text{in } \Omega,\\
  \,\,\,\quad u=\eta &  \text{on } \partial \Omega
\end{cases}
\end{equation*}
is understood in the sense that
$$
   \int_\Omega u (-\Delta)\xi   dx = \int_\Omega \xi d\mu - \int_{\partial \Omega}\frac{\partial \xi}{\partial \overrightarrow{n}_x}d \eta \quad \text{ for all $\xi \in C^{2}_0(\overline{\Omega})$}.
$$
Here $\overrightarrow{n}_x$ denotes the unit inward normal vector  at $x \in \partial \Omega$. This inspired the definition of solutions to \eqref{concentrated1}. In effect, motivated by \cite{CH}, we define the normal derivative $\frac{\partial^{\alpha}\eta}{\partial \overrightarrow{n}_x^{\alpha}}$ in the distributional sense as follows
$$
   \left\langle \frac{\partial^{\alpha}\eta}{\partial \overrightarrow{n}_x^{\alpha}}, \xi \right\rangle = \int_{\partial \Omega}\frac{\partial^{\alpha}\xi}{\partial \overrightarrow{n}_x^{\alpha}}(x)  d \eta(x) \qquad \xi \in \mathbb{X}_\alpha(\Omega),
$$
where for $x \in \partial \Omega$
$$
   \frac{\partial^{\alpha}\xi}{\partial \overrightarrow{n}_x^{\alpha}}(x) \defeq \lim_{t \to 0}\frac{\xi(x+t\overrightarrow{n}_x)-\xi(x)}{t^{\alpha}}= \lim_{t \to 0^{+}}t^{-\alpha}\xi(x+t\overrightarrow{n}_x).
$$
Roughly speaking, the derivative $\frac{\partial^{\alpha}\eta}{\partial \overrightarrow{n}_x^{\alpha}}$ may be approximated by measures $\left\lbrace t^{-\alpha}\eta_t \right\rbrace_{t>0}$ with support in the level sets
$$
 \Xi(t) =  \left\lbrace x \in \Omega: \rho(x)= t \right\rbrace,
$$
where
$$
\rho(x) \defeq \left\{
\begin{array}{lc}
  \text{dist}(x, \partial \Omega) & \forall\,\,x \in \Omega \\
  0 & \text{otherwise}.
\end{array}
\right.
$$
\begin{defi}
A function $u \in L^{1}(\Omega)$ is a weak solution to \eqref{concentrated1} if
$$
   \int_\Omega u (-\Delta)^{\alpha} \xi dx = \int_{\partial \Omega}\frac{\partial^{\alpha}\xi}{\partial \overrightarrow{n}_x^{\alpha}}(x)  d \eta(x) \qquad \text{ for all }\xi \in \mathbb{X}_\alpha(\Omega).
$$
\end{defi}

For convenience of the reader, we will provide some facts from \cite{CH}. Firstly, the approximation of $\frac{\partial^{\alpha}\eta}{\partial \overrightarrow{n}_x^{\alpha}}$ by Radon measures concentrated on manifolds in $\Omega$ is done as follows: by  \cite{GT} and \cite{Marcus-V1}, there exists $t_0 > 0$ such that
$$
   \Omega_t = \left\lbrace x \in \Omega: \rho(x) > t \right\rbrace,
$$
is a $C^{2}$ domain for all $t \in [0, t_0]$ and for each point $x_t \in \partial \Omega_t$ there corresponds $x \in \partial \Omega$ such that
$$
   |x-x_t|= \rho(x).
$$
Conversely, for each $x \in \partial \Omega$, there is a unique $x_t \in \partial \Omega_t$ so that
$$
  |x-x_t|= \rho(x_t).
$$
In this way, for each Borel subset $E \subset \partial \Omega$, there is a unique $E_t \subset \partial \Omega_t$ so that $E_t=\left\lbrace x_t: x \in E\right\rbrace$. Define the measures
$$
  \eta_t(E_t)=\eta(E),\,\text{for each} \,\,\,E_t \subset \partial \Omega_t \,\,\,\text{Borel.}
$$
Hence $\eta_t$ is a Radon measure supported in $\partial \Omega_t$ that may be extended to $\overline{\Omega}$ by
$$
  \eta_t(E)= \eta_t(E \cap \partial \Omega_t), \qquad E \subset \overline{\Omega}\text{ Borel}.
$$
By \cite[Proposition 2.1]{CH}  we have that the Radon measures $\left\lbrace t^{-\alpha }\eta_t \right\rbrace_{t>0}$ converge to $\frac{\partial^{\alpha}\eta}{\partial \overrightarrow{n}_x^{\alpha}}$ in the distributional sense
$$
   \lim_{t \to 0^{+}}\int_{\partial \Omega_t}\xi(x)t^{-\alpha}d\eta_t(x)=\int_{\partial \Omega}\frac{\partial \xi}{\partial \overrightarrow{n}_x^{\alpha}}d\eta(x) \qquad \text{ for all }\xi \in \mathbb{X}_\alpha(\Omega).
$$
Consequently, since $t^{-\alpha}\eta_t$ has support in $\Omega$, the solution of problem \eqref{concentrated1}  may be approximated by  solutions $u_t$ to

\begin{equation}\label{concentrated12}
\begin{cases}
(-\Delta)^\alpha u=t^{-\alpha}\eta_t &  \text{in } \Omega,\\
  \qquad \quad u=0 &  \text{on } \mathbb{R}^{N}\setminus \Omega.
\end{cases}
\end{equation}

For existence of solutions to \eqref{concentrated12} and their convergence to a solution of \eqref{concentrated1} we refer the reader to \cite{CH}. We now give the definition of solution to problem \eqref{maineq-conc}.

\begin{defi}
A function $u \in L^{1}(\Omega)$, with $g(x, |\nabla u|)\in L^{1}(\Omega)$, is a weak solution to \eqref{maineq-conc} if
$$
  \int_\Omega u (-\Delta)^{\alpha}\xi dx = \int_\Omega g(x, |\nabla u|)\xi dx + \sigma \int_\Omega \xi d \nu- \int_{\partial \Omega}\frac{\partial \xi}{\partial \overrightarrow{n}_x^{\alpha}}d\eta(x)  - \varrho\int_{\mathbb{R}^{N}\setminus \Omega}\mathcal{N}_\alpha \xi d \mu
$$
for all $\xi \in \mathbb{X}_\alpha(\Omega).$\end{defi}

\begin{teo}\label{MThm3}Assume that $p \in (0, p^{\ast})$ and that $g$ satisfies \eqref{asump g all p}. Moreover, suppose that all the general hypothesis from Section \ref{prelim} are in force and that $\eta \in \mathcal{M}(\overline{\Omega})$ is supported in $\partial \Omega$. Then, problem \eqref{maineq-conc} admits a solution $u$ for all small enough $c$ as in \eqref{asump g all p}.
\end{teo}

\begin{proof}
Firstly, we solve \eqref{concentrated1}. Consider $k_0$ large enough so that $k \geq k_0$ implies $t_k\defeq \frac{1}{k} < t_0$. In what follows, we take $k \geq k_0$.  Let $\eta_k \in C^{1}(\overline{\Omega})$ be  non-negative, with
$$
   \text{supp}(\eta_k) \subset (\partial \Omega)_{2t_k}\defeq \left\lbrace x\in \overline{\Omega}: \rho(x)< 2t_k \right\rbrace$$and $\eta_k \to \eta$ in the sense of duality in the Banach space
$$
   C_\alpha(\overline{\Omega})\defeq \left\lbrace f \in C(\overline{\Omega}): \rho^{-\alpha}f \in  C(\overline{\Omega}) \right\rbrace.
$$
By \cite[Lemma 2.1]{CV1}, $\mathbb{X}_\alpha(\Omega) \subset C_\alpha(\overline{\Omega})$. Moreover, by  Banach-Steinhaus theorem (or Uniform boundedness Principle), it is possible to derive that
$$
   \|\eta_k \|_{L^{1}(\Omega)},\,\|t_k^{-\alpha}\eta_k \|_{L^{1}(\Omega)}  \leq C \quad  \text{for all}\,\,\, k.
$$
For convenience of the reader,  we next provide details in deriving the uniform boundedness of  $t_k^{-\alpha}\eta_k$ in $L^{1}(\Omega)$ fashion (a similar argument works for $\eta_k$). Observe that for all $\xi \in C_\alpha(\overline{\Omega})$ there is a constant $ C(\xi) >0$ so that
$$
   \Big\vert \int_\Omega t_k^{-\alpha}\eta_k \xi dx\Big\vert \leq \int_{ (\partial \Omega)_{2t_k}} t_k^{-\alpha}\eta_k |\xi| d\eta(x) \leq 2^{\alpha}\int_{ (\partial \Omega)_{2t_k}} |\rho^{-\alpha}\xi|\eta_k d\eta(x) \leq C(\xi) \,\,\,\text{for all}\,\,\, k.
$$
Here we used the fact that $\eta_k$ is uniformly bounded in $L^{1}(\Omega)$. By the Uniform Boundedness Principle, there is a constant $C>0$ (independent of $k$) so that
\begin{equation}\label{UPB}
\sup_{||\xi||_{C(\overline{\Omega})} \leq 1}\Big\vert \int_\Omega t_k^{-\alpha}\eta_k \xi dx\Big\vert \leq C \quad\text{ for all }k.
\end{equation}
For each $k$, consider a sequence of compact sets $K_{k, n}\nearrow \Omega$ as $n \to \infty$ and take $\xi_{n, k} \in C_\alpha(\overline{\Omega})$ such that $0\leq \xi_{n, k} \leq 1$, $\xi_{n, k}= 1$ in $K_{n, k}$ and $\xi_{n, k}=0$ near $\partial \Omega.$ Hence monotone convergence theorem and \eqref{UPB}  imply
$$
   \int_\Omega t_k^{-\alpha}\eta_kdx=\lim_{n \to \infty}\int_\Omega t_k^{-\alpha}\eta_k \chi_{K_{n, k}}dx \leq \limsup_{n \to \infty} \int_{\Omega}t_k^{-\alpha}\eta_k \xi_{n, k}dx\leq C.
$$
In this way, we conclude that $t_k^{-\alpha}\eta_k$ is uniformly integrable in $\Omega$.

 Moreover, by \cite[Section 6]{CH} there is a non-negative solution $w_k \in L^{1}(\Omega)$ of \eqref{concentrated12} with $t^{-\alpha}\eta_t$ replaced by $t_k^{-\alpha}\eta_k$. By \cite[Proposition 2.3]{CV}, for each $q \in (1, p^{\ast})$, there is a constant $c_0 > 0$ such that
$$
   ||w_k||_{W_0^{1, q}(\Omega)} \leq c_0||t_k^{-\alpha}\eta_k||_{L^{1}(\Omega)} \leq C \text{ for all }k.
$$
Therefore, $u_k$ converges weakly to a function $w \in W^{1, q}_0(\Omega)$. Moreover, in  \cite{CH} it is shown the convergence of $w_k$ in $L^{1}(\Omega)$ to a solution of \eqref{concentrated1}. Hence, the limiting profile $w$ must be the solution of \eqref{concentrated1} and we deduce the further regularity $w \in W^{1, q}_0(\Omega)$. We call $w=\mathbb{G}_\alpha\left[\frac{\partial \eta}{\partial \overrightarrow{n}_x^{\alpha}}\right]$.
Consider
$$
   u= \mathbb{G}_\alpha[g(x, |\nabla u|)+ \sigma \nu]+ \varrho \mathbb{P}_\alpha[\mu]+ \mathbb{G}_\alpha\left[\frac{\partial \eta}{\partial \overrightarrow{n}_x^{\alpha}}\right],
$$
where $v=\mathbb{G}_\alpha[g(x, |\nabla u|)+\sigma\nu]$ solves
\begin{equation}\label{diri34}
\begin{cases}
(-\Delta)^\alpha v=g\left(x, \left|\nabla\left( v+\varrho\mathbb{P}_{\alpha}[\mu]+ \mathbb{G}_\alpha\left[\frac{\partial \eta}{\partial \overrightarrow{n}_x^{\alpha}}\right]\right)\right|\right) + \sigma \nu &  \text{on } \Omega,\\
\qquad \quad v=0 &  \text{in } \R^N\setminus\Omega
\end{cases}
\end{equation}
in the sense of Definition \ref{weak solu1}. The existence of $v$ is achieved as in Section \ref{subcritical} recalling that $\mathbb{G}_\alpha\left[\frac{\partial \eta}{\partial \overrightarrow{n}_x^{\alpha}}\right] \in W^{1, p}_0(\Omega)$. This ends the proof of the theorem.
\end{proof}

\begin{remark} As in the first part of paper, we are able to prove uniqueness assertions to \eqref{maineq-conc} provided $g$ fulfils the assumptions of the Comparison Principle result (Theorem \ref{CompPrin}).
\end{remark}

\section{Final comments}\label{ClosRem}

We have presented some existence/uniqueness and regularity results for problems driven by fractional diffusion operators and with nonlinear gradient sources and measures. In order to conclude our work, let us point out that our results also work for problems with more general nonlinear gradient term $g \in C^0(\Omega\times \R \times [0, \infty)) \cap L^{1}(\Omega)$ as follows:
$$
\left\{
\begin{array}{rcll}
  (-\Delta)^\alpha u(x) & = &g(x, u, |\nabla u|) +\sigma \nu &  \text{in } \Omega,\\
  u(x) & = & \varrho\mu &  \text{in } \R^N\setminus\Omega,
\end{array}
\right.
$$
which fulfils the following growth assumptions: $|g(x, s,  \xi)|\leq c_0|\xi|^p+c_1|s|^q+\varepsilon|f(x)|,$
where $0<q<+\infty$ and $p$, $\sigma$, $\nu$, $\varrho$, $\mu$ and $f$ are as before. In particular, such analysis extends the former results in \cite{VC1}.

Finally, it is worth emphasising that our approach can be applied for weak solutions of more general fractional-type problems (as long as existence and compactness of the associated Green operators hold) of the form:
$$
\left\{
\begin{array}{rcll}
  -\mathcal{L}_{\Phi} u(x) & = &g(x, u, |\nabla u|) +\sigma \nu &  \text{in } \Omega,\\
  u(x) & = & \varrho\mu &  \text{in } \R^N\setminus\Omega,
\end{array}
\right.
$$
where $-\mathcal{L}_{\Phi}$ is a nonlocal elliptic operator defined by
$$
 \displaystyle \langle -\mathcal{L}_{\Phi} u(x), \varphi\rangle = \int_{\R^N} \int_{\R^N} \Phi(u(x)-u(y))(\varphi(x)-\varphi(y))\mathrm{K}(x, y)dxdy,
$$
for every smooth function $\varphi$ with compact support. Also, the function $\Phi: \R \to \R$ is assumed to be continuous, fulfilling $\Phi(0) = 0$ and
the monotonicity property
$$
\lambda|s|^2\leq \Phi(s)s\leq \Lambda |s|^2 \quad \forall\,\,\, s\in \R
$$
for constants $\Lambda \geq \lambda >0$, and $\mathrm{K}: \R^N \times \R^N \to \R$ is a general singular kernel satisfying the following structural properties: there exist constants $\Lambda \geq \lambda >0$ and $\mathfrak{M}, \varsigma >0$ such that
\begin{enumerate}
  \item[(1)][{\bf Symmetry}] $\mathrm{K}(x, y) = \mathrm{K}(y, x)$ for all $x, y \in \R^N$;
  \item[(2)][{\bf Elipticity condition}] $\lambda \leq \mathrm{K}(x, y).|x-y|^{N+2\alpha} \leq \Lambda$ for $x, y \in \R^N$, $x \neq y$;
  \item[(3)][{\bf Integrability at infinity}] $0\leq \mathrm{K}(x, y) \leq \frac{\mathfrak{M}}{|x-y|^{N+\varsigma}}$ for $x \in B_2$ and $y \in \R^N \setminus B_{\frac{1}{4}}$.
  \item[(4)][{\bf Translation invariance}] $\mathrm{K}(x+z, y+z)= \mathrm{K}(x, y)$ for all $x,y,z \in \R^N$, $x \neq y$.
  \item[(5)][{\bf Continuity}] The map $x \mapsto \mathrm{K}(x, y)$ is continuous in $\R^N \setminus \{y\}$.
\end{enumerate}
Clearly the former class of operators have as prototype the $\alpha-$fractional Laplacian operator provided $\Phi(s) = s$ and $\mathrm{K}(x, y) = \frac{1}{|x-y|^{N+2\alpha}}$ (cf. \cite{KMS18} and references therein).

Finally, for $0< \alpha < 1< p< \infty$ consider the nonlinear
integro-differential operator
$$
\displaystyle (-\Delta)_p^{\alpha} u(x) \defeq  \mathrm{C}_{N,s, p}.\text{p.v.}\int_{\R^n} \frac{|u(x)-u(y)|^{p-2}(u(x)-u(y))}{|x-y|^{N+\alpha p}}dy.
$$
Such an operator is nowadays known as \textit{fractional
$p-$Laplacian} (see \cite{daSR19}, \cite{daSS19-1}, \cite{KMS15} and
\cite{KMS18} and references therein). It seems an interesting and
challenging proposal to seek new strategies in order to prove
existence/uniqueness and regularity results, for instance, without
the explicit representation formula for solutions and associated
Green function.

\noindent{\bf Acknowledgements.} J.V. S. was partially supported by Coordena\c{c}\~{a}o de Aperfei\c{c}oamento de Pessoal de N\'{i}vel Superior (PNPD/CAPES-UnB-Brazil)  Grant No. 88887.357992/2019-00 and CNPq-Brazil under Grant No. 310303/2019-2. A.S. is supported by 
PICT 2017-0704, by Universidad Nacional de San Luis under grants
PROIPRO 03-2418 and PROICO 03-1916.  P. O. is supported by Proyecto Bienal  B080 Tipo 1 (Res. 4142/2019-R).

\end{document}